\documentclass{article}

\usepackage[left=3.5cm,right=3.5cm,top=3.0cm,bottom=3.0cm]{geometry}

\usepackage{lineno,hyperref}
\usepackage{authblk}
\usepackage{amsmath,amsfonts,amssymb,amsthm}
\usepackage[utf8x]{inputenc}
\usepackage[dvipsnames,table]{xcolor}
\usepackage{graphicx}
\usepackage{subcaption}
\providecommand{\keywords}[1]{\textbf{\textit{keywords:}} #1}

\usepackage{tikz}

\theoremstyle{plain}% Theorem-like structures provided by amsthm.sty
\newtheorem{theorem}{Theorem}[section]
\newtheorem{lemma}[theorem]{Lemma}
\newtheorem{corollary}[theorem]{Corollary}

\theoremstyle{definition}

\newtheorem{example}[theorem]{Example}

\theoremstyle{remark}
\newtheorem{remark}{Remark}

\begin{document}

\title{A finite difference approximation of a two dimensional
time fractional advection-dispersion problem}
% \author{Author's Name}

% \author[1]{Alice Smith}\providecommand{\keywords}[1]{\textbf{\textit{Index terms---}} #1}
% \author[2]{Bob Jones}
% \affil[1]{Department of Mathematics, University X}
% \affil[2]{Department of Biology, University Y}

% \author[1]{Daniel Kopta}
% \author[1]{Konstantin Shkurko}
% \author[1,2,3]{Josef Spjut}
% \author[1]{Erik Brunvand}
% \author[1]{Al Davis}
% \affil[1]{School of Computing, University of Utah, Salt Lake City, UT, USA\\
%        \{redacted\}@email.edu}
% \affil[2]{NVIDIA Research}
% \affil[3]{Department of Engineering, Harvey Mudd College, Claremont, CA, USA}

\author[1]{Carlos E. Mej\'ia \thanks{cemejia@unal.edu.co}}
\author[2]{Alejandro Piedrahita H. \thanks{alejandro.piedrahita@udea.edu.co}}
% \author[1,2]{Author E\thanks{E.E@university.edu}}
\affil[1]{Escuela de Matem\'aticas, Universidad Nacional de Colombia, Medell\'in, Colombia}
\affil[2]{Instituto de Matem\'aticas, Universidad de Antioquia, Medell\'in, Colombia}

\date{}

\maketitle

\begin{abstract}
The main purpose of this paper is the construction and analysis of an implicit finite difference scheme for the numerical solution of a two dimensional time-fractional advection-dispersion equation with variable coefficients. The dispersion term is in nondivergence form and the fractional derivative is taken in the sense of Caputo. Equations of this sort are potentially useful as models of contaminant transport in groundwater. Provided some mild assumptions are satisfied, proofs of consistency, stability and convergence are obtained. Furthermore, we offer a general but simple framework for the matrices required in computations and everything is tested by a well selected set of numerical experiments.
\end{abstract}


%\chapter{First Chapter}

\section{Introduction}

Fractional derivatives are associated with memory and  hereditary properties of materials and processes. They are known since the seventeenth century but only recently have become an important subject of applied mathematics. They might be applied on time and/or space variables and are suitable for a variety of topics, for instance, the behavior of viscoelastic materials (\cite{book:diethelm}) and the anomalous diffusion of a contaminant in porous media (\cite{Fomin2010}). For these and other uses of fractional derivatives the reader is invited to consult \cite{book:podlubny,OldhamS2006,book:kilbasST2006}.

There are a variety of fractional derivatives, i.e. Caputo, Riemann-Liouville, Grünwald-Letnikov and many others. Moreover, the fractional derivatives  can be single-term or multi-term, according to the number of differentiation orders which can be real or complex numbers. Our interest is on the modeling of transport phenomena in porous media through a two dimensional time-fractional advection-dispersion equation with variable coefficients in which the diffusion term is given in nondivergence form and the differentiation order is a real number between $0$ and $1$.

Many authors have proposed numerical solutions for time-fractional differential equations. For instance, \cite{ZhuangLiu2007} introduces a two dimensional single term time-fractional diffusion equation with variable coefficients and diffusion term in nondivergence form. The present work owes several ideas to this paper. Some authors solve one dimensional time-fractional differential equations with diffusion term in nondivergence form. We mention \cite{WangRen2018}, in which the interest is in the direct problem and \cite{MP2017} which solves a one dimensional time-fractional diffusion equation as a tool in the  process of solving an inverse problem.

Among the authors who face two dimensional time-fractional differential equations, we mention \cite{Balasim-Ali2017}, in which the problem is similar to ours but the coefficients are constant and \cite{EcheverryMejia2018}, which deals with a two dimensional inverse source problem and introduces a particular case of our numerical scheme for the necessary solution of the direct problem.

In this paper, for the implicit approximation of the Caputo fractional derivative we implement the known scheme very well described in \cite{LinXu2007}. This scheme appears elsewhere, for instance, in \cite{ZhuangLiu2007}. For the advection and dispersion terms we implement standard central finite difference schemes. 

The rest of the paper is divided in three sections. Section \ref{2} defines the equation and the numerical methods. The next section contains the consistency, stability and convergence statements along with their proofs. The numerical experiments and final remarks are presented in Section \ref{4} .

\section{The problem and the numerical method\label{2}}
The prediction of the environmental consequences of groundwater contamination
is an important goal for researchers. Our interest is to help in this
prediction through a numerical approximation of a mathematical model based on
a partial differential equation known as an advection-dispersion equation.
Our equation has variable coefficients and a time
fractional derivative rather than the classical time derivative. Other
features of our model are: It considers the contaminant transport through a
two dimensional porous medium with variable advection and dispersion function
coefficients given by two components each. Moreover, the diffusion terms are in nondivergence form. 
For this matter we follow references \cite{ZhuangLiu2007,qiao2017rbf,zhao2017numerical}. All of them show that 
nondivergence diffusion terms are worth and with Caputo time-fractional derivatives provide useful 
models of anomalous diffusion.

% Moreover, the diffusion terms are in 
% nondivergence form. For this matter we follow references like \cite{Liu_Z_A_T_B2007}. 
% All of them show that nondivergence diffusion terms are worth and with Caputo time-fractional 
% derivatives provide useful non-local models of anomalous diffusion.

\subsection{The initial-boundary value problem}\label{conditions}
We consider the two-dimensional initial-boundary value problem
\begin{linenomath*}
	\begin{multline}\label{eq:tfade}
	% \begin{split}
	u^{(\alpha)}_t(x,y,t) + a(x,y,t)u_x(x,y,t) + b(x,y,t)u_y(x,y,t) \\ = c(x,y,t) u_{xx}(x,y,t) + d(x,y,t) u_{yy}(x,y,t) + f(x,y,t)
	% \end{split}
	\end{multline}
\end{linenomath*}
with initial condition
\begin{linenomath*}
	\begin{equation}\label{cond:initial}
	u(x,y,0) = \psi(x,y), \qquad (x,y)\in\Omega := (x_L,x_R)\times (y_L,y_R)\subset \mathbb{R}^2,
	\end{equation}
\end{linenomath*}
and Dirichlet boundary condition
\begin{linenomath*}
	\begin{equation}\label{cond:boundary}
	u(x,y,t) = 0, \qquad  (x,y)\in\partial\Omega\times (0,T],
	\end{equation}
\end{linenomath*}
where: %$\Omega = (x_L,x_R)\times (y_L,y_R)\subset \mathbb{R}^2$, $0<t\le T$ and

\begin{enumerate}
	\item $u\left(  x,y,t\right)  $ is the contaminant concentration.
	
	\item $c$ is the longitudinal dispersion variable coefficient.
	
	\item $d$ is the transversal dispersion variable coefficient.
	
	\item $a$ and $b$ are the longitudinal and transversal advection coefficients
	respectively. They are basically the seepage or average pore water velocity
	and if one of the directions is predominant, only one of the advection
	function coefficients is nonzero.
	
	\item $f$ is a known source or sink term.
	
	\item $u_{t}^{\left(  \alpha\right)  }$ is the Caputo fractional derivative of
	order $\alpha$ given by%
	\begin{linenomath*}
		\begin{equation}
		u_{t}^{(\alpha)}(x,y,t):=%
		\begin{cases}
		{\displaystyle\frac{1}{\Gamma(1-\alpha)}\int_{0}^{t}\frac{u_{t}(x,y,\xi
				)}{(t-\xi)^{\alpha}}\,d\xi,} & 0<\alpha<1,\\[3mm]%
		u_{t}(x,y,t), & \alpha=1,
		\end{cases}
		\label{caputo}%
		\end{equation}
	\end{linenomath*}
	
	\item The variable coefficient functions $a,b,c,d$ satisfy the following
	uniform bounds: There are two positive constants $A$ and $D$ so that
	\begin{linenomath*}
		\begin{equation}\label{bounds}
		\begin{array}
		[c]{c}%
		0\leq a(x,y,t) \leq A,\\[1mm]%
		0\leq b(x,y,t) \leq A,
		\end{array}
		\qquad\mbox{and}\qquad%
		\begin{array}
		[c]{c}%
		0<D\leq c(x,y,t),\\[1mm]%
		0<D\leq d(x,y,t).
		\end{array}
		\end{equation}
	\end{linenomath*}
\end{enumerate}

The next subsection deals with the proposed finite difference approximation. 

\subsection{The numerical scheme}

Let the mesh points $x_{i}=x_{L}+i\Delta x$, $0\leq i\leq N_{x}$, $y_{j}%
=y_{L}+j\Delta y$, $0\leq j\leq N_{y}$ and $t_{k}=k\Delta t$, $0\leq k\leq
N_{t}$, where $\Delta x=(x_{R}-x_{L})/N_{x}$ and $\Delta y=(y_{R}-y_{L}%
)/N_{y}$ are the spatial grid sizes in the $x$- and $y$-direction,
respectively, and $\Delta t=T/N_{t}$ is the time step size. The values of the
functions $u,a,b,c,d$ and $f$ at the grid points are denoted by $u_{i,j}%
^{k}=u(x_{i},y_{j},t_{k})$, $a_{i,j}^{k}=a(x_{i},y_{j},t_{k})$, $b_{i,j}%
^{k}=b(x_{i},y_{j},t_{k})$, $c_{i,j}^{k}=c(x_{i},y_{j},t_{k})$, $d_{i,j}%
^{k}=d(x_{i},y_{j},t_{k})$ and $f_{i,j}^{k}=f(x_{i},y_{j},t_{k})$,
respectively. The initial condition is set as $u_{i,j}^{0}=\psi_{i,j}%
=\psi(x_{i},y_{j})$. The Dirichlet boundary condition at $x=x_{L}$ is set as
$u_{0,j}^{k}=0$ and similarly on the other three sides of the boundary.
%The second inequality can be proved in the same way.

The Caputo fractional derivative at time $t_{k+1}$ is approximated by
\begin{linenomath*}
	\begin{equation}
	u_{t}^{(\alpha)}(x_{i},y_{j},t_{k+1})=\sigma_{\alpha,\Delta t}\sum_{s=0}%
	^{k}\omega_{s}^{(\alpha)}\left(  u_{i,j}^{k-s+1}-u_{i,j}^{k-s}\right)
	+O\left((\Delta t)^{2-\alpha}\right), \label{eq:caputo-aprox}%
	\end{equation}
\end{linenomath*}
for $k=0,\hdots,N_{t}-1$, where $\sigma_{\alpha,\Delta t}=\dfrac{1}{(\Delta t)^{\alpha
	}\Gamma(2-\alpha)}$ and $\omega_{s}^{(\alpha)}=(s+1)^{1-\alpha}-s^{1-\alpha}$
for $s=0,\hdots,N_{t}$, as described in \cite{LinXu2007}.

Derivatives with respect to space coordinates
are approximated by central difference formulae. Let
$v_{i,j}^{k}$ be the numerical approximation to $u_{i,j}^{k}$. The discrete
version of \eqref{eq:tfade} is the implicit finite difference scheme (\textbf{IFDS}) 
given by

\begin{linenomath*}
	\begin{multline}\label{ifds}
	\sigma_{\alpha,\Delta t}\sum_{s=0}^{k}\omega_{s}^{(\alpha)}\left(
	v_{i,j}^{k-s+1}-v_{i,j}^{k-s}\right)  +a_{i,j}^{k+1}\,\frac{v_{i+1,j}%
		^{k+1}-v_{i-1,j}^{k+1}}{2\Delta x}+b_{i,j}^{k+1}\,\frac{v_{i,j+1}%
		^{k+1}-v_{i,j-1}^{k+1}}{2\Delta y}\\
	=c_{i,j}^{k+1}\,\frac{v_{i+1,j}^{k+1}-2v_{i,j}^{k+1}+v_{i-1,j}^{k+1}}{(\Delta
		x)^{2}}+d_{i,j}^{k+1}\,\frac{v_{i,j+1}^{k+1}-2v_{i,j}^{k+1}+v_{i,j-1}^{k+1}%
	}{(\Delta y)^{2}}+f_{i,j}^{k+1},
	\end{multline}
\end{linenomath*}
for $i=1,\hdots,N_{x}-1$, $j=1,\hdots,N_{y}-1$ and $k=0,\hdots,N_{t}-1$. 

\subsection{Consistency} 

In order to prove consistency of scheme (\textbf{IFDS}), it is convenient to denote \eqref{eq:tfade} by 
\begin{linenomath*}
	\[
	S(u)=S\left( \partial _{t},\partial _{x},\partial _{y},\partial
	_{xx},\partial _{yy}\right) u=f(x,y,t),
	\]%
\end{linenomath*}
where 
\begin{linenomath*}
	\begin{eqnarray*}
		S\left( u\right)  &=&u_{t}^{(\alpha
			)}(x,y,t)+a(x,y,t)u_{x}(x,y,t)+b(x,y,t)u_{y}(x,y,t) \\
		&&-c(x,y,t)u_{xx}(x,y,t)-d(x,y,t)u_{yy}(x,y,t)
	\end{eqnarray*}
\end{linenomath*}

Likewise, we establish the following alternative notation for scheme (\textbf{IFDS}) 
\begin{linenomath*}
	\[
	S_{\Delta}(v) = S_{\Delta t,\Delta x, \Delta y }(v_{i,j}^{k+1}),
	\]%
\end{linenomath*}
where 
\begin{linenomath*}
	\begin{multline*}
	S_{\Delta}(v) = \sigma_{\alpha,\Delta t}\sum_{s=0}^{k}\omega_{s}^{(\alpha)}\left(
	v_{i,j}^{k-s+1}-v_{i,j}^{k-s}\right)  +a_{i,j}^{k+1}\,\frac{v_{i+1,j}%
		^{k+1}-v_{i-1,j}^{k+1}}{2\Delta x}+b_{i,j}^{k+1}\,\frac{v_{i,j+1}%
		^{k+1}-v_{i,j-1}^{k+1}}{2\Delta y}\\
	-c_{i,j}^{k+1}\,\frac{v_{i+1,j}^{k+1}-2v_{i,j}^{k+1}+v_{i-1,j}^{k+1}}{(\Delta
		x)^{2}}-d_{i,j}^{k+1}\,\frac{v_{i,j+1}^{k+1}-2v_{i,j}^{k+1}+v_{i,j-1}^{k+1}%
	}{(\Delta y)^{2}}.
	\end{multline*}
\end{linenomath*}
for $i=1,\hdots,N_{x}-1$, $j=1,\hdots,N_{y}-1$ and $k=0,\hdots,N_{t}-1$. It is known that 
if $u$ is a smooth function, then at interior points of its domain the following equalities hold:%
\begin{linenomath*}
	\begin{eqnarray*}
		u_{t}^{\left( \alpha \right) }(x_i,y_j,t_k) - \sigma _{\alpha ,\Delta
			t}\sum_{s=0}^{k}\omega _{s}^{(\alpha )}\left(
		u_{i,j}^{k-s+1}-u_{i,j}^{k-s}\right)  &=&O\left((\Delta t)^{2-\alpha}\right) \\
		u_{x}(x_i,y_j,t_k) - \frac{u_{i+1,j}^{k+1}-u_{i-1,j}^{k+1}}{2\Delta x} &=&O\left( \left(
		\Delta x\right) ^{2}\right) \\
		u_{y}(x_i,y_j,t_k) - \,\frac{u_{i,j+1}^{k+1}-u_{i,j-1}^{k+1}}{2\Delta y} &=&O\left( \left(
		\Delta y\right) ^{2}\right) \\
		u_{xx}(x_i,y_j,t_k) - \frac{u_{i+1,j}^{k+1}-2u_{i,j}^{k+1}+u_{i-1,j}^{k+1}}{(\Delta x)^{2}}
		&=&O\left( \left( \Delta x\right) ^{2}\right)  \\
		u_{yy}(x_i,y_j,t_k) - \frac{u_{i,j+1}^{k+1}-2u_{i,j}^{k+1}+u_{i,j-1}^{k+1}}{(\Delta y)^{2}}
		&=&O\left( \left( \Delta y\right) ^{2}\right) 
	\end{eqnarray*}
\end{linenomath*}

Let us denote the addition of all right hand sides above by 
\begin{linenomath*}
	\begin{equation}
	O\left(\Delta\right) = O\left(  (\Delta t)^{2-\alpha},\left(  \Delta x\right)  ^{2},\left(  \Delta y\right)
	^{2}\right). \label{bigodel} 
	\end{equation}
\end{linenomath*}

Thus,
\begin{linenomath*}
	\begin{equation}
	S\left(  u\right)  -S_{\Delta}\left(  u\right)  =O\left(  (\Delta t)^{2-\alpha},\left(
	\Delta x\right)  ^{2},\left(  \Delta y\right)  ^{2}\right)  \label{bigototal}%
	\end{equation}
\end{linenomath*}
and we have proved

\begin{lemma}
	The finite difference scheme (\textbf{IFDS}) is consistent with the 
	partial differential equation \eqref{eq:tfade}. 
\end{lemma}

Other way to write (\ref{bigototal}) is%
\begin{linenomath*}
	\begin{equation}
	f_{i,j}^{k+1}-S_{\Delta}\left(  u\right)  =O\left(  (\Delta t)^{2-\alpha},\left(  \Delta
	x\right)  ^{2},\left(  \Delta y\right)  ^{2}\right)  .\label{fs1}%
	\end{equation}
\end{linenomath*}

By setting
\begin{linenomath*}
	\begin{equation}
	\mu_{1}=\frac{(\Delta t)^{\alpha}}{2\Delta x},\quad\mu_{2}=\frac{(\Delta
		t)^{\alpha}}{2\Delta y},\quad\mu_{3}=\frac{(\Delta t)^{\alpha}}{(\Delta
		x)^{2}},\quad\mu_{4}=\frac{(\Delta t)^{\alpha}}{(\Delta y)^{2}},\quad
	\tau=\frac{1}{\sigma_{\alpha,\Delta t}}=(\Delta t)^{\alpha}\Gamma(2-\alpha),\label{mu}%
	\end{equation}
\end{linenomath*}
and
\begin{linenomath*}
	\begin{equation}%
	\begin{split}
	p_{i,j}^{k} &  =\Gamma(2-\alpha)\left[  \mu_{3}c_{i,j}^{k}-\mu_{1}a_{i,j}%
	^{k}\right]  ,\quad q_{i,j}^{k}=\Gamma(2-\alpha)\left[  \mu_{3}c_{i,j}%
	^{k}+\mu_{1}a_{i,j}^{k}\right]  ,\\
	r_{i,j}^{k} &  =\Gamma(2-\alpha)\left[  \mu_{4}d_{i,j}^{k}-\mu_{2}b_{i,j}%
	^{k}\right]  ,\quad h_{i,j}^{k}=\Gamma(2-\alpha)\left[  \mu_{4}d_{i,j}^{k}%
	+\mu_{2}b_{i,j}^{k}\right]  ,\\
	e_{i,j}^{k} &  =1+p_{i,j}^{k}+q_{i,j}^{k}+r_{i,j}^{k}+h_{i,j}^{k}%
	=1+2\,\Gamma(2-\alpha)\left[  \mu_{3}c_{i,j}^{k}+\mu_{4}d_{i,j}^{k}\right]  ,
	\end{split}
	\label{coe}%
	\end{equation}
\end{linenomath*}
we split scheme (\textbf{IFDS}) in two stages:
\begin{enumerate}
	\item
	For $k=0,$ it is
	\begin{linenomath*}
		\begin{equation}%
		\begin{split}
		-\left(  p_{i,j}^{1}v_{i+1,j}^{1}+q_{i,j}^{1}v_{i-1,j}^{1}\right)
		+e_{i,j}^{1}v_{i,j}^{1} &  -\left(  r_{i,j}^{1}v_{i,j+1}^{1}+h_{i,j}%
		^{1}v_{i,j-1}^{1}\right)  \\
		&  =v_{i,j}^{0}+\tau f_{i,j}^{1}%
		\end{split}
		\label{scheme0}%
		\end{equation}
	\end{linenomath*}
	for $i=1,\hdots,N_{x}-1$ and $j=1,\hdots,N_{y}-1$
	\item 
	For $k=1,\hdots,N_{t}%
	-1$, the scheme is
	\begin{linenomath*}
		\begin{equation}%
		\begin{split}
		-\left(  p_{i,j}^{k+1}v_{i+1,j}^{k+1}+q_{i,j}^{k+1}v_{i-1,j}^{k+1}\right)   &
		+e_{i,j}^{k+1}v_{i,j}^{k+1}-\left(  r_{i,j}^{k+1}v_{i,j+1}^{k+1}+h_{i,j}%
		^{k+1}v_{i,j-1}^{k+1}\right)  \\
		&  =v_{i,j}^{k}-\sum_{s=1}^{k}\omega_{s}^{(\alpha)}\left(  v_{i,j}%
		^{k-s+1}-v_{i,j}^{k-s}\right)  +\tau f_{i,j}^{k+1}%
		\end{split}
		\label{schemek}%
		\end{equation}
	\end{linenomath*}
	where $i=1,\hdots,N_{x}-1$ and $j=1,\hdots,N_{y}-1$.
\end{enumerate}

The next lemma provides the main features of the quadrature weights $\omega_{s}^{(\alpha)}.$

\begin{lemma}
	The quadrature weights $\omega_{s}^{(\alpha)}$ are positive and $\omega_{s}^{(\alpha)} > \omega_{s+1}^{(\alpha)} $ for all $s=0,1,\dots$
\end{lemma}

The nonnegativity of all variable coefficients of scheme (\ref{scheme0})-(\ref{schemek}) is a desirable feature. Definitions of $q$ and $h$ in (\ref{coe}) establish that they are nonnegative functions. For the other coefficients,  nonnegativity is achieved provided a mild assumption on the grid sizes is imposed. The details are in the following lemma.

\begin{lemma}\label{bounds:grids}  
	If the variable coefficients $a, b, c $ and $d$ satisfy the bounds \eqref{bounds} and $\max\{\Delta x,\Delta y\}\leq2D/A$, then $p_{i,j}%
	^{k}\geq0$ and $r_{i,j}^{k}\geq0$ for each $i=1,\hdots,N_{x}-1$%
	,\ $j=1,\hdots,N_{y}-1$ and $k=1,\hdots,N_{t}-1$.
\end{lemma}

\begin{proof}
	The proof consists on the following straightforward computations:
	\begin{linenomath*}
		\begin{equation*}
		p_{i,j}^{k}=\Gamma\left(  2-\alpha\right)  \frac{\left(  \Delta t\right)
			^{\alpha}}{(\Delta x)^{2}}\left[  c_{i,j}^{k}-\frac{\Delta x}{2}a_{i,j}%
		^{k}\right]  \geq\frac{\tau}{(\Delta x)^{2}}\left[  D-\frac{\Delta x}%
		{2}A\right]  \geq0\label{qg0}%
		\end{equation*}
	\end{linenomath*}
	and
	\begin{linenomath*}
		\begin{equation*}
		r_{i,j}^{k}=\frac{\tau}{(\Delta y)^{2}}\left[  d_{i,j}^{k}-\frac{\Delta y}%
		{2}b_{i,j}^{k}\right]  \geq\frac{\tau}{(\Delta y)^{2}}\left[  D-\frac{\Delta
			y}{2}A\right]  \geq0.\label{hg0} \qedhere
		\end{equation*}
	\end{linenomath*}
\end{proof}

\subsection{The linear system}

Let 
\begin{linenomath*}
	\begin{equation}
	v^{k} = \left[v_{\ast,1}^k \ \ v_{\ast,2}^k \ \cdots \ v_{\ast,N_{y}-1}^k \right]^T
	\label{vijk} 
	\end{equation}
\end{linenomath*}
where $v^k_{\ast,j} = \left[v_{1,j}^k \ \ v_{2,j}^k \ \cdots \ v_{N_{x}-1,j}^k \right]^T$, \quad
$j=1,\hdots, N_{y}-1$.

In the particular case $N_x=N_y=N$, the $(N-1)^2$ equations \eqref{scheme0}-\eqref{schemek} may be written in matrix form
\begin{linenomath*}
	\begin{equation}\label{eq:sys}
	A^{(k+1)} v^{k+1} = y^k 
	\end{equation}
\end{linenomath*}
for each $0\le k<N_t$, where $A^{(k)}$ is the $(N-1)^2\times (N-1)^2$ matrix of coefficients resulting 
from the system of difference equations at the gridpoints at level $t=t_k$,
$v^{k} = \left[v_{\ast,1}^k \ \ v_{\ast,2}^k \ \cdots \ v_{\ast,N-1}^k \right]^T$
with $v^k_{\ast,j} = \left[v_{1,j}^k \ \ v_{2,j}^k \ \cdots \ v_{N-1,j}^k \right]^T$, 
%for $j=1,\hdots, N_y-1$; 
and $y^{k} = \left[y_{\ast,1}^k \ \ y_{\ast,2}^k \ \cdots \ y_{\ast,N-1}^k \right]^T$
with
\begin{linenomath*}
	\begin{equation*}
	y^{k}_{\ast,j} =
	\begin{cases}
	\psi_{\ast,j} + \tau f_{\ast,j}^1,  &  k=0,\\
	\gamma v_{\ast,j}^{1} + \gamma \psi_{\ast,j} + \tau f_{\ast,j}^2,  &  k=1,\\
	\displaystyle \gamma v_{\ast,j}^{k} + \sum_{s=1}^{k-1} \left( \omega_s^{(\alpha)} - \omega_{s+1}^{(\alpha)} \right) v_{\ast,j}^{k-s} + \omega_k^{(\alpha)}\psi_{\ast,j} + \tau f_{\ast,j}^{k+1},  &   1< k< N_t,\\
	\end{cases}
	\end{equation*}
\end{linenomath*}
where $\psi_{\ast,j} = \left[\psi_{1,j} \ \ \psi_{2,j} \ \cdots \ \psi_{N-1,j} \right]^T$,
$f^k_{\ast,j} = \left[f_{1,j}^k \ \ f_{2,j}^k \ \cdots \ f_{N-1,j}^k \right]^T$ and
$\gamma = \left( 2 - 2^{1-\alpha} \right)$.

Eq. \eqref{eq:sys} requires, at each time step, to solve a linear system where 
the right-hand side $y^k$ utilizes all the history of the computed 
solution up to that time, and $A^{(k)}$ is a band matrix with a block structure. 
Each block is a $(N-1)\times(N-1)$ matrix and together they give $A^{(k)}$ the following form

\begin{linenomath*}
	\newcolumntype{C}[1]{@{}>{\rule[0.5\dimexpr-#1+1.2ex]{0pt}{#1}\hfil$}p{#1}<{$\hfil}@{}}
	\begin{equation}\label{matrix:A}
	A^{(k)} =
	\left[
	\begin{array}{ C{1.2cm} !{\color{gray!50}\vrule} C{1.2cm} !{\color{gray!50}\vrule} C{1.2cm} !{\color{gray!50}\vrule} C{1.2cm} !{\color{gray!50}\vrule} C{1.2cm} }
	T_{1}^k & D_{1}^k  &    0     &     \cdots    &   0       \\ \arrayrulecolor{gray!50}\hline
	\widetilde{D}_{1}^k & T_{2}^k  & D_{2}^k   &    \ddots     &  \vdots        \\ \arrayrulecolor{gray!50}\hline
	0    & \ddots & \ddots  & \ddots  &  0        \\ \arrayrulecolor{gray!50}\hline
	\vdots    &   \ddots     & \widetilde{D}_{N-3}^k & T_{N-2}^k & D_{N-2}^k  \\ \arrayrulecolor{gray!50}\hline
	0    &  \cdots      &    0     & \widetilde{D}_{N-2}^k & T_{N-1}^k
	\end{array}  
	\right].
	\end{equation}
\end{linenomath*}

In this expression each $T_{\ell}^k$ is a tridiagonal matrix given by
\begin{linenomath*}
	\[
	T_{\ell}^{(k)} =
	\begin{bmatrix}
	e_{1,\ell}^k & -p_{1,\ell}^k & 0       & \cdots & 0 \\[2pt]
	-q_{2,\ell}^k & e_{2,\ell}^k & -p_{2,\ell}^k  & \ddots & \vdots \\[5pt]
	0     & \ddots & \ddots & \ddots  & 0 \\[5pt]
	\vdots & \ddots  & -q_{N-2,\ell}^k & e_{N-2,\ell}^k & -p_{N-2,\ell}^k \\[2pt]
	0      & \cdots & 0      & -q_{N-1,\ell}^k & e_{N-1,\ell}^k
	\end{bmatrix},
	\]
\end{linenomath*}
while $D_{\ell}^k=[d_{i,j}^k]$ and $\widetilde{D}_{\ell}^k=[\widetilde{d}_{i.j}^k]$ are 
diagonal matrices defined by $d_{i,i}^k=-r_{i,\ell}^k$ and ${\widetilde{d}_{i,i}}^k=-h_{i,\ell}^k$, 
for $i=1,\hdots,N-1$.

\begin{remark}\label{diagonal}
	Note that for each $1\leq i \leq(N-1)^2$, there exists exactly one $1\leq \ell_i\leq N-1$ such that
	the resulting diagonal entry $A_{i,i}^{(k)}$ of \eqref{matrix:A} is determined by
	\begin{linenomath*}
		\begin{equation*}%\label{caputo}
		A_{i,i}^{(k)} := e_{i,\ell_i}^k = 1 + p_{i,\ell_i}^k + q_{i,\ell_i}^k + r_{i,\ell_i}^k + h_{i,\ell_i}^k.
		\end{equation*}
	\end{linenomath*}
	The off-diagonal entries $A_{i,j}^{(k)}$ with $i\ne j$, can be determined in the same way.
\end{remark}

\section{The approximation}\label{3}

In this section we prove the unconditional stability and the convergence of scheme \eqref{scheme0}-\eqref{schemek}. 
Both results are inspired by \cite{ZhuangLiu2007}. Let $v^k$ be given by \eqref{vijk} for $k=0,\dots,N_{t}.$
\subsection{Stability}
\begin{theorem} \label{stability}
	If the hypotheses of lemma \ref{bounds:grids} hold, scheme \eqref{scheme0}-\eqref{schemek} 
	for the homogeneous ($f \equiv 0$) initial-boundary value 
	problem \eqref{eq:tfade}-\eqref{cond:initial}-\eqref{cond:boundary} is unconditionally stable.
\end{theorem}
\begin{proof}\
	\begin{enumerate}
		\item Scheme \eqref{scheme0}: Let $\lVert v^{1}\rVert_{\infty}=\left\vert v_{l,m}^{1}\right\vert =\max_{i,j}\left\vert v_{i,j}^{1}\right\vert .$ 
		We show that $\lVert v^{1} \rVert_{\infty} \leq \lVert v^0 \rVert_{\infty} $
		\begin{linenomath*}
			\begin{align*}
			\left\vert v_{l,m}^{1}\right\vert  & =\left(  1+p_{l,m}^{1}+q_{l,m}^{1}%
			+r_{l,m}^{1}+h_{l,m}^{1}\right)  \left\vert v_{l,m}^{1}\right\vert \\
			& -\left(  p_{l,m}^{1}\left\vert v_{l,m}^{1}\right\vert +q_{l,m}^{1}\left\vert
			v_{l,m}^{1}\right\vert \right)  -\left(  r_{l,m}^{1}\left\vert v_{l,m}^{1}\right\vert +h_{l,m}^{1}\left\vert
			v_{l,m}^{1}\right\vert \right)  \\
			& \leq\left(  1+p_{l,m}^{1}+q_{l,m}^{1}+r_{l,m}^{1}+h_{l,m}^{1}\right)  \left\vert
			v_{l,m}^{1}\right\vert \\
			& -\left(  p_{l,m}^{1}\left\vert v_{l+1,m}^{1}\right\vert +q_{l,m}^{1}\left\vert
			v_{l-1,m}^{1}\right\vert \right) -\left(  r_{l,m}^{1}\left\vert v_{l,m+1}^{1}\right\vert +h_{l,m}^{1}\left\vert
			v_{l,m-1}^{1}\right\vert \right)  \\
			& \leq\left\vert \left(  1+p_{l,m}^{1}+q_{l,m}^{1}+r_{l,m}^{1}+h_{l,m}^{1}\right)
			v_{l,m}^{1}\right.  \\
			& -\left(  p_{l,m}^{1}v_{l+1,m}^{1}+q_{l,m}^{1}v_{l-1,m}^{1}\right) \left.  -\left(  r_{l,m}^{1}v_{l,m+1}^{1}+h_{l,m}^{1}v_{l,m-1}^{1}\right)
			\right\vert  \\
			& =\left\vert v_{l,m}^{0}\right\vert \leq \lVert v^0 \rVert_{\infty}.
			\end{align*}
		\end{linenomath*}
		
		\item Scheme \eqref{schemek}: The proof is by induction over $k$. Suppose $\lVert v^n \rVert_{\infty} \leq \lVert v^0 \rVert_{\infty} $
		for $n=2,\dots,k.$ We prove the inequality $\lVert v^{k+1} \rVert_{\infty} \leq \lVert v^0 \rVert_{\infty}.$
		From now on, we will write $\omega_{s}$ rather than $\omega_{s}^{\left(  \alpha\right)  }.$ Notice that
		the right hand side in scheme \eqref{schemek} is%
		\begin{linenomath*}
			\begin{align*}
			v_{i, j}^{k} - \sum_{s=1}^{k} \omega_{s}\left(  v_{i,j}^{k-s+1}-v_{i,j}^{k-s}\right) 
			& =\omega_{0}v_{i,j}^{k}-\omega_{1}v_{i,j}^{k}+\omega_{1}v_{i,j}^{k-1}-\omega _{2}v_{i,j}^{k-1}+\omega_{2}v_{i,j}^{k-2} \\
			& -\ldots-\omega_{k}v_{i,j}^{1}+\omega_{k}v_{i,j}^{0}\\
			& =\left( \omega_{0}-\omega_{1}\right)  v_{i,j}^{k} + \left(  \omega_{1}-\omega_{2}\right)  v_{i,j}^{k-1} + \left( \omega_{2}-\omega_{3}\right)  v_{i,j}^{k-2} \\
			&  + \dots + \left(  \omega_{k-1}-\omega_{k}\right)  v_{i,j}^{1} + \omega_{k}v_{i,j}^{0}. 
			\end{align*}
		\end{linenomath*}
		Now we look at the left hand side. Let $\lVert v^{k+1} \rVert_{\infty}=\left\vert v_{l,m}^{k+1}\right\vert = \max_{i,j}\left\vert v_{i,j}^{k+1}\right\vert.$ 
		\begin{linenomath*}
			\begin{align*}
			\left\vert v_{l,m}^{k+1}\right\vert  & =\left(  1+p_{l,m}^{k+1}+q_{l,m}^{k+1}+r_{l,m}^{k+1}+h_{l,m}^{k+1}\right)  \left\vert v_{l,m}^{k+1}\right\vert \\
			& -\left(  p_{l,m}^{k+1}\left\vert v_{l,m}^{k+1}\right\vert +q_{l,m}^{k+1}\left\vert v_{l,m}^{k+1}\right\vert \right)  
			-\left(  r_{l,m}^{k+1}\left\vert v_{l,m}^{k+1}\right\vert +h_{l,m}^{k+1}\left\vert v_{l,m}^{k+1}\right\vert \right) \\
			& \leq\left(  1+p_{l,m}^{k+1}+q_{l,m}^{k+1}+r_{l,m}^{k+1}+h_{l,m}^{k+1}\right) \left\vert v_{l,m}^{k+1}\right\vert \\
			& -\left(  p_{l,m}^{k+1}\left\vert v_{l+1,m}^{k+1}\right\vert +q_{l,m}^{k+1}\left\vert v_{l-1,m}^{k+1}\right\vert \right)
			-\left(  r_{l,m}^{k+1}\left\vert v_{l,m+1}^{k+1}\right\vert +h_{l,m}^{k+1}\left\vert v_{l,m-1}^{k+1}\right\vert \right) \\
			& \leq\left\vert \left(  1+p_{l,m}^{k+1}+q_{l,m}^{k+1}+r_{l,m}^{k+1}+h_{l,m}^{k+1}\right)  v_{l,m}^{k+1}\right.  \\
			& -\left(  p_{l,m}^{k+1}v_{l+1,m}^{k+1}+q_{l,m}^{k+1}v_{l-1,m}^{k+1}\right) 
			\left.  -\left(  r_{l,m}^{k+1}v_{l,m+1}^{k+1}+h_{l,m}^{k+1}v_{l,m-1}^{k+1}\right)  \right\vert  \\
			& =\left\vert \left(  \omega_{0}-\omega_{1}\right)  v_{l,m}^{k} + \dots + \left(  \omega_{k-1}-\omega_{k}\right)  v_{l,m}^{1} + \omega_{k}v_{l,m}^{0} \right\vert \\
			& \leq \left(  \omega_{0}-\omega_{1}\right) \left\Vert  v^{k} \right\Vert_{\infty} + \dots + \left(  \omega_{k-1}-\omega_{k}\right)  \left\Vert v^{1} \right\Vert_{\infty} + 
			\omega_{k} \left\Vert v^{0} \right\Vert_{\infty} \\
			& \leq \left(  \omega_{0}-\omega_{1}\right) \left\Vert  v^{0} \right\Vert_{\infty} + \dots + \left(  \omega_{k-1}-\omega_{k}\right)  \left\Vert v^{0} \right\Vert_{\infty} + 
			\omega_{k} \left\Vert v^{0} \right\Vert_{\infty} \\
			% & = \left(  \omega_{0}-\omega_{k-1}\right) \left\Vert  v^{0} \right\Vert_{\infty}\\
			& = {\left\Vert v^{0}\right\Vert}_{\infty}. \qedhere
			\end{align*}
		\end{linenomath*}
	\end{enumerate}
	
\end{proof}

This theorem allows us to prove an additional stability bound. Let $v_{ij}^{0}$
and $\tilde{v}_{ij}^{0}$ be the initial discrete values corresponding to two
initial conditions $\psi _{ij}$ and $\tilde{\psi}_{ij}.$ We may think of two
different measurements of the initial concentration. Furthermore, let $%
v_{ij}^{k}$ and $\tilde{v}_{ij}^{k}$ be the corresponding discrete
approximations obtained by the numerical schemes \eqref{scheme0} and \eqref{schemek}.  Let $%
\varepsilon _{ij}^{k}=v_{ij}^{k}-\tilde{v}_{ij}^{k}$ and 
\begin{linenomath*}
	\begin{equation}
	E^{k}=\left[ \varepsilon _{\ast ,1}^{k}\ \ \varepsilon _{\ast ,2}^{k}\
	\cdots \ \varepsilon _{\ast ,N_{y}-1}^{k}\right] ^{T}  \label{eijk}
	\end{equation}%
\end{linenomath*}
where $\varepsilon _{\ast ,j}^{k}=\left[ \varepsilon _{1,j}^{k}\ \
\varepsilon _{2,j}^{k}\ \cdots \ \varepsilon _{N_{x}-1,j}^{k}\right] ^{T}$,
\quad $j=1,\dots ,N_{y}-1$.

\begin{corollary}
	If the hypotheses of lemma \ref{bounds:grids} are satisfied, the numerical errors induced by 
	initial-value conditions in scheme \eqref{scheme0}-\eqref{schemek} for the inhomogeneous initial-boundary value 
	problem \eqref{eq:tfade}-\eqref{cond:initial}-\eqref{cond:boundary} do not propagate. More precisely, they satisfy 
	the bound
	\begin{linenomath*}
		\[
		\left\Vert E^{k}\right\Vert _{\infty }\leq \left\Vert E^{0}\right\Vert_{\infty },\qquad k=1,2,\ldots. 
		\]
	\end{linenomath*}
\end{corollary}

\subsection{Convergence}

Stability and convergence proofs follow similar patterns. Let $%
\varepsilon _{ij}^{k}=u_{ij}^{k}-v_{ij}^{k}$ and 
\begin{linenomath*}
	\begin{equation}
	E^{k}=\left[ \varepsilon _{\ast ,1}^{k}\ \ \varepsilon _{\ast ,2}^{k}\
	\cdots \ \varepsilon _{\ast ,N_{y}-1}^{k}\right] ^{T}  \label{erijk}
	\end{equation}%
\end{linenomath*}
where $\varepsilon _{\ast ,j}^{k}=\left[ \varepsilon _{1,j}^{k}\ \
\varepsilon _{2,j}^{k}\ \cdots \ \varepsilon _{N_{x}-1,j}^{k}\right] ^{T}$,
\quad $j=1,\dots ,N_{y}-1$.
The convergence of the scheme is given by the following theorem.

\begin{theorem}
	If the hypotheses of lemma \ref{bounds:grids} hold, then % there is a constant $C$ independent of $\Delta t, \Delta x $ and $\Delta y$ so that 
	\begin{linenomath*}
		\begin{equation}\label{convergence}
		\left\Vert E^{k}\right\Vert _{\infty }\leq \lVert E^{0} \rVert_{\infty} + \left(\Delta t\right)^{\alpha} O\left(\Delta \right),
		\qquad k=1,2,\ldots 
		\end{equation}
	\end{linenomath*}
	where $O\left(\Delta \right)$ is defined by \eqref{bigodel}.
\end{theorem}
\begin{proof}
	The proof is by induction on $k.$
	\item Case $k=1.$ We show that 
	\begin{linenomath*}
		\[
		\lVert E^{1} \rVert_{\infty} \leq  \lVert E^{0} \rVert_{\infty} + \left(\Delta t\right)^{\alpha} O\left(\Delta \right).
		\]
	\end{linenomath*}
	Let $\lVert E^{1}\rVert_{\infty}=\left\vert \varepsilon_{lm}^{1}\right\vert
	=\max_{i,j}\left\vert \varepsilon_{i,j}^{1}\right\vert$. In this case the scheme under consideration is \eqref{scheme0}.
	\begin{linenomath*}
		\begin{align*}
		\left\vert \varepsilon_{l,m}^{1}\right\vert  &= \left(  1+p_{l,m}^{1}+q_{l,m}^{1} + r_{l,m}^{1}+h_{l,m}^{1}\right)  \left\vert \varepsilon_{l,m}^{1}\right\vert 
		-\left(  p_{l,m}^{1}\left\vert \varepsilon_{l,m}^{1}\right\vert +q_{l,m}^{1}\left\vert \varepsilon_{l,m}^{1}\right\vert \right)  
		-\left(  r_{l,m}^{1}\left\vert \varepsilon_{l,m}^{1}\right\vert +h_{l,m}^{1}\left\vert \varepsilon_{l,m}^{1}\right\vert \right)  \\
		& \leq\left(  1+p_{l,m}^{1}+q_{l,m}^{1}+r_{l,m}^{1}+h_{l,m}^{1}\right)  \left\vert \varepsilon_{l,m}^{1}\right\vert 
		-\left(  p_{l,m}^{1}\left\vert \varepsilon_{l+1,m}^{1}\right\vert +q_{l,m}^{1}\left\vert \varepsilon_{l-1,m}^{1}\right\vert \right)  \\
		& -\left(  r_{l,m}^{1}\left\vert \varepsilon_{l,m+1}^{1}\right\vert +h_{l,m}^{1}\left\vert \varepsilon_{l,m-1}^{1}\right\vert \right)  \\
		& \leq\left\vert \left( 1+p_{l,m}^{1}+q_{l,m}^{1}+r_{l,m}^{1}+h_{l,m}^{1}\right) \varepsilon_{l,m}^{1}\right.  
		-\left(  p_{l,m}^{1}\varepsilon_{l+1,m}^{1}+q_{l,m}^{1}\varepsilon_{l-1,m}^{1}\right)  \\
		& \left.  -\left(  r_{l,m}^{1}\varepsilon_{l,m+1}^{1}+h_{l,m}^{1}\varepsilon_{l,m-1}^{1}\right) \right\vert \\
		& = \left\vert \left(  1+p_{l,m}^{1}+q_{l,m}^{1}+r_{l,m}^{1}+h_{l,m}^{1}\right) u_{l,m}^{1} - \left(  p_{l,m}^{1}u_{l+1,m}^{1}+q_{l,m}^{1}u_{l-1,m}^{1} \right) \right.\\
		&\left. -\left(  r_{l,m}^{1} u_{l,m+1}^{1}+h_{l,m}^{1}u_{l,m-1}^{1}\right) \right. -\left(  1+p_{l,m}^{1}+q_{l,m}^{1}+r_{l,m}^{1}+h_{l,m}^{1}\right) v_{l,m}^{1}  \\
		& + \left(  p_{l,m}^{1}v_{l+1,m}^{1}+q_{l,m}^{1}v_{l-1,m}^{1} \right) \left. +\left(  r_{l,m}^{1}v_{l,m+1}^{1}+h_{l,m}^{1}v_{l,m-1}^{1}\right) \right\vert  \\
		& =\left\vert u_{l,m}^1 + \tau \left( S_{\Delta}\left(u_{l,m}^{1}\right) -\sigma_{\alpha,\Delta t}\left(u_{l,m}^1-u_{l,m}^0\right)\right) \right. 
		-\left(  1+p_{l,m}^{1}+q_{l,m}^{1}+r_{l,m}^{1}+h_{l,m}^{1}\right) v_{l,m}^{1} \\
		& +\left(  p_{l,m}^{1}v_{l+1,m}^{1}+q_{l,m}^{1}v_{l-1,m}^{1} \right)  \left. +\left(  r_{l,m}^{1}v_{l,m+1}^{1}+h_{l,m}^{1}v_{l,m-1}^{1}\right) \right\vert  \\
		& = \left\vert u_{l,m}^0 + \tau \left( S\left(u_{l,m}^{1}\right) + O\left(\Delta\right)\right) \right.
		-\left(  1+p_{l,m}^{1}+q_{l,m}^{1}+r_{lm}^{1}+h_{l,m}^{1}\right) v_{l,m}^{1} \\
		& +\left(  p_{l,m}^{1}v_{l+1,m}^{1}+q_{l,m}^{1}v_{l-1,m}^{1} \right) + \left(  r_{l,m}^{1}v_{l,m+1}^{1}+h_{l,m}^{1}v_{l,m-1}^{1}\right) \vert \\
		& = \left\vert u_{l,m}^0 + \tau f_{l,m}^1 + \tau\, O\left(\Delta\right)-v_{l,m}^0 - \tau f_{lm}^1 \right\vert \\
		& = \left\vert \varepsilon_{l,m}^0 + \tau\, O\left(\Delta\right) \right\vert \\
		& \leq \left\Vert E^{0} \right\Vert_{\infty} + (\Delta t)^{\alpha}\, O\left(\Delta\right).
		\end{align*}
	\end{linenomath*}
	Now suppose
	\begin{linenomath*}
		\[
		\left\Vert E^{s} \right\Vert_{\infty} \leq \left\Vert E^{0} \right\Vert_{\infty} + (\Delta t)^{\alpha}\, O\left(\Delta\right)
		\]
	\end{linenomath*}
	holds for $s=1,2,\ldots,k.$ We prove the result for $s=k+1$. 
	Let $\lVert E^{k+1}\rVert_{\infty}=\left\vert \varepsilon_{lm}^{k+1}\right\vert
	=\max_{i,j}\left\vert \varepsilon_{ij}^{k+1}\right\vert$.
	By the same argument as before, 
	
	\begin{linenomath*}
		\begin{align*}
		\left\vert \varepsilon_{lm}^{k+1}\right\vert  
		& \leq \left\vert \left(  1+p_{l,m}^{k+1}+q_{l,m}^{k+1}+r_{l,m}^{k+1}+h_{l,m}^{k+1}\right) u_{l,m}^{k+1} 
		-\left(  p_{l,m}^{k+1}u_{l+1,m}^{k+1}+q_{l,m}^{k+1}u_{l-1,m}^{k+1} \right) \right.\\
		& \left. -\left(  r_{l,m}^{k+1} u_{l,m+1}^{k+1}+h_{l,m}^{k+1} u_{l,m-1}^{k+1}\right) \right. \\
		& \left. -\left(  1+p_{l,m}^{k+1} + q_{l,m}^{k+1}+r_{l,m}^{k+1}+h_{l,m}^{k+1}\right) v_{l,m}^{k+1} 
		+\left(  p_{l,m}^{k+1}v_{l+1,m}^{k+1}+q_{l,m}^{k+1}v_{l-1,m}^{k+1} \right) \right. \\
		& \left. +\left(  r_{l,m}^{k+1}v_{l,m+1}^{k+1}+h_{l,m}^{k+1}v_{l,m-1}^{k+1}\right) 
		\right\vert .
		\end{align*}
	\end{linenomath*}
	In this case the scheme is \eqref{schemek} and as before, we write $\omega_{s}$ instead of $\omega_{s}^{(\alpha)}.$ 
	The last expression becomes
	\begin{linenomath*}
		\begin{align*}
		& \left\vert u_{l,m}^{k+1} + \tau \left( S_{\Delta}\left(u_{l,m}^{k+1}\right) 
		- \sigma_{\alpha,\Delta t} \sum_{s=0}^{k}\omega _{s}\left( u_{l,m}^{k-s+1}-u_{l,m}^{k-s}\right)\right) \right. \\
		& \left. -v_{l,m}^{k} + \sum_{s=1}^{k}\omega _{s}\left( v_{l,m}^{k-s+1}-v_{l,m}^{k-s}\right) -\tau f_{l,m}^{k+1} \right\vert \\
		& = \left\vert u_{l,m}^{k+1} + \tau S\left(u_{l,m}^{k+1}\right) + \tau O\left(\Delta \right)
		- \sum_{s=0}^{k}\omega _{s}\left( u_{l,m}^{k-s+1}-u_{l,m}^{k-s}\right) \right. \\
		& \left. -v_{l,m}^{k} + \sum_{s=1}^{k}\omega _{s}\left(
		v_{l,m}^{k-s+1}-v_{l,m}^{k-s}\right) -\tau f_{l,m}^{k+1} \right\vert \\
		& = \left\vert u_{l,m}^{k+1} + \tau\, f_{l,m}^{k+1} + \tau O\left(\Delta \right)
		- \sum_{s=0}^{k}\omega _{s}\left( u_{l,m}^{k-s+1}-u_{l,m}^{k-s}\right) \right. \\
		& \left. -v_{l,m}^{k} + \sum_{s=1}^{k}\omega _{s}\left(
		v_{l,m}^{k-s+1}-v_{l,m}^{k-s}\right) -\tau f_{l,m}^{k+1} \right\vert \\
		& = \left\vert \sum_{s=1}^{k} \left(\omega _{s-1} - \omega_{s} \right) u_{l,m}^{k-s+1} 
		+ \omega_{k} u_{l,m}^{0} + \tau\, O\left(\Delta\right) \right. \left.  -v_{l,m}^{k} + \sum_{s=1}^{k}\omega _{s}\left(v_{l,m}^{k-s+1}-v_{l,m}^{k-s}\right)   \right\vert \\
		& = \left\vert \sum_{s=0}^{k-1} \left(\omega _{s} - \omega_{s+1} \right) \varepsilon_{l,m}^{k-s}
		+ \omega_{k} \varepsilon_{l,m}^{0} + \tau\, O\left(\Delta\right) \right\vert \\
		& \leq \sum_{s=0}^{k-1} \left(\omega _{s} - \omega_{s+1} \right) \left\vert \varepsilon_{l,m}^{k-s} \right\vert
		+ \omega_{k} \left\vert \varepsilon_{l,m}^{0} \right\vert + \tau\, O\left(\Delta\right) \\
		& \leq \sum_{s=0}^{k-1} \left(\omega _{s} - \omega_{s+1} \right) \left\Vert E^{k-s} \right\Vert_{\infty}
		+ \omega_{k} \left\Vert E^{0} \right\Vert_{\infty} + \tau\, O\left(\Delta\right) \\
		& \leq \sum_{s=0}^{k-1} \left(\omega _{s} - \omega_{s+1} \right) \left\Vert E^{0} \right\Vert_{\infty}
		+ \omega_{k} \left\Vert E^{0} \right\Vert_{\infty} + \tau\, O\left(\Delta\right) \\
		%   & = \sum_{s=0}^{k-2} \left(\omega _{s} - \omega_{s+1} \right) \left\Vert E^{0} \right\Vert_{\infty}
		%   + \tau O\left(\Delta\right) \\
		%   & = \left(\omega _{0} - \omega_{k-1} \right) \left\Vert E^{0} \right\Vert_{\infty} + \tau O\left(\Delta\right) \\
		& = \left\Vert E^{0} \right\Vert_{\infty} + (\Delta t)^{\alpha}\, O\left(\Delta\right). \qedhere \\
		\end{align*}
	\end{linenomath*}
\end{proof}

\section{Numerical experiments and final remarks}\label{4}
In order to demonstrate the reliability of our numerical method, three examples are presented. 
The absolute errors in the approximation $v$ of $u$ at time $t=t_k$ are measured by the maximum norm
\begin{linenomath*}
	\[
	{\left\Vert v^{k} - u(t_k) \right\Vert}_{\infty} := \max_{i,j} \left| v_{i,j}^{k} - u_{i,j}^{k} \right|.
	\]
\end{linenomath*}
\begin{example}\label{ex:1}
	We consider the time fractional advection-dispersion equation
	\begin{linenomath*}
		\begin{multline*}
		% \begin{split}
		u^{(\alpha)}_t(x,y,t) + a(x,y,t)u_x(x,y,t) + b(x,y,t)u_y(x,y,t) \\ = c(x,y,t) u_{xx}(x,y,t) 
		+ d(x,y,t) u_{yy}(x,y,t) + f(x,y,t)
		% \end{split}
		\end{multline*}
	\end{linenomath*}
	on a finite square domain $\Omega = (0,1)\times (0,1)$ for $0\leq t\leq 1$,
	with the initial condition
	\begin{linenomath*}
		\begin{equation*}
		u(x,y,0) = \sin\pi x \sin\pi y, \qquad (x,y)\in\Omega,
		\end{equation*}
	\end{linenomath*}
	\begin{table}[t]
		\centering		
		\par
		\renewcommand{\arraystretch}{1.3}
		\begin{tabular}{p{2cm}p{2cm}p{3cm}p{2cm}}
			\hline\hline
			\rowcolor{gray!15}
			$\Delta t$ & $\Delta x = \Delta y$ &  Max. error & Order \\ %\hline
			\hline
			\rowcolor{gray!15}
			$\alpha = 0.1$   &         &           &   \\  
			\hline
			$1/16$           & $1/4$   & 1.440e-01 & - - - - \\ 
			$1/32$           & $1/8$   & 4.070e-02 & 1.823   \\ 
			$1/64$           & $1/16$  & 1.043e-02 & 1.964  \\ 
			$1/128$          & $1/32$  & 2.607e-03 & 2.001  \\
			$1/256$          & $1/64$  & 6.530e-04 & 1.997  \\ 
			\hline
			\rowcolor{gray!15}
			$\alpha = 0.5$   &         &           &   \\  
			\hline
			$1/16$           & $1/4$   & 1.415e-01 & - - - - \\ 
			$1/32$           & $1/8$   & 4.055e-02 & 1.803  \\ 
			$1/64$           & $1/16$  & 1.045e-02 & 1.957  \\ 
			$1/128$          & $1/32$  & 2.625e-03 & 1.993  \\
			$1/256$          & $1/64$  & 6.627e-04 & 1.986  \\  
			\hline
			\rowcolor{gray!15}
			$\alpha = 0.9$ &  &    &   \\  
			\hline
			$1/16$           & $1/4$   & 1.588e-01 & - - - - \\ 
			$1/32$           & $1/8$   & 4.434e-02 & 1.841   \\ 
			$1/64$           & $1/16$  & 1.189e-02 & 1.899   \\ 
			$1/128$          & $1/32$  & 3.365e-03 & 1.821  \\
			$1/256$          & $1/64$  & 1.053e-03 & 1.676   \\ 
			\hline\hline
		\end{tabular}%
		\caption{Absolute errors and order of convergence at $t=1$ for Example \ref{ex:1}.}		
		\label{table:ex1} 
	\end{table}
	and the boundary condition
	\begin{linenomath*}
		\begin{equation*}
		u(x,y,t) = 0, \qquad  (x,y)\in\partial\Omega\times (0,1].
		\end{equation*}
	\end{linenomath*}
	The advection and dispersion coefficients are given by
	\begin{linenomath*}
		\begin{equation*}
		a(x,y,t) = \frac{1}{\sin\pi y}, \qquad b(x,y,t) = \frac{1}{\sin\pi x}, \\
		\end{equation*}
	\end{linenomath*}
	and
	\begin{linenomath*}
		\begin{equation*} 
		c(x,y,t) = \frac{x}{\pi^2 \Gamma(3-\alpha)(t^2+1)}, \qquad d(x,y,t) = \frac{y}{\pi^2 \Gamma(3-\alpha)(t^2+1)},
		\end{equation*}
	\end{linenomath*}
	respectively, the source or sink function is 
	\begin{linenomath*}
		\begin{equation*}
		f(x,y,t) = \dfrac{1}{\Gamma(3-\alpha)} \left( 2t^{2-\alpha} + x + y \right) \sin\pi x\sin\pi y + \pi(t^2+1)( \cos\pi x + \cos\pi y )
		\end{equation*}
	\end{linenomath*}
	and the exact concentration is
	\begin{linenomath*}
		\begin{equation*}
		u(x,y,t) = (t^2+1)\sin\pi x \sin\pi y.
		\end{equation*}
	\end{linenomath*}

	Numerical experiments for fractional derivatives of orders $\alpha=0.1, \alpha=0.5$ and $\alpha=0.9$
	are listed in Table \ref{table:ex1}. The columns for absolute errors and order of convergence are the main features 
	of this table. This is a somewhat extreme example due to the fact that the advection coefficients 
	$a$ and $b$ do not safisfy the bounds \eqref{bounds} 
	
\end{example}

\begin{table}[t]
	\centering	
	\par
	\renewcommand{\arraystretch}{1.3}
	\begin{tabular}{llcccccc}
		\hline\hline
		\rowcolor{gray!15}
		$\Delta t$ & $\Delta x = \Delta y$ &  \multicolumn{3}{c}{Absolute error} & \multicolumn{3}{c}{Order of convergence} \\ %\hline
		\hline
		\rowcolor{gray!15}
		$\alpha = 0.1$   &         & $\varepsilon=$1e-1 & $\varepsilon=$1e-3 & $\varepsilon=$1e-5 & $\varepsilon=$1e-1 & $\varepsilon=$1e-3 & $\varepsilon=$1e-5       \\  
		\hline
		$1/16$           & $1/4$   & 1.103e-01          & 1.431e-01          & 1.451e-01          & - - -              & - - -              & - - -   \\ 
		$1/32$           & $1/8$   & 2.982e-02          & 5.154e-02          & 5.622e-02          & 1.887              & 1.473              & 1.368    \\ 
		$1/64$           & $1/16$  & 7.703e-03          & 1.448e-02          & 1.711e-02          & 1.953              & 1.831              & 1.716    \\ 
		$1/128$          & $1/32$  & 1.927e-03          & 3.386e-03          & 4.484e-03          & 1.999              & 2.097              & 1.932    \\
		\rowcolor{gray!15}
		\hline
		$\alpha = 0.5$   &         & $\varepsilon=$1e-1 & $\varepsilon=$1e-3 & $\varepsilon=$1e-5 & $\varepsilon=$1e-1 & $\varepsilon=$1e-3 & $\varepsilon=$1e-5       \\  
		\hline
		$1/16$           & $1/4$   & 1.108e-01          & 1.445e-01          & 1.468e-01          & - - -              &  - - -             &  - - - \\ 
		$1/32$           & $1/8$   & 3.001e-02          & 5.079e-02          & 5.451e-02          & 1.884              & 1.508              & 1.430     \\ 
		$1/64$           & $1/16$  & 7.817e-03          & 1.406e-02          & 1.617e-02          & 1.941              & 1.853              & 1.753     \\ 
		$1/128$          & $1/32$  & 1.974e-03          & 3.169e-03          & 4.090e-03          & 1.986              & 2.150              & 1.983     \\
		\rowcolor{gray!15}
		\hline
		$\alpha = 0.9$   &         & $\varepsilon=$1e-1 & $\varepsilon=$1e-3 & $\varepsilon=$1e-5 & $\varepsilon=$1e-1 & $\varepsilon=$1e-3 & $\varepsilon=$1e-5       \\  
		\hline
		$1/16$           & $1/4$   & 1.228e-01          & 1.451e-01          & 1.470e-01          & - -                &  - -               &  - -   \\ 
		$1/32$           & $1/8$   & 3.377e-02          & 4.401e-02          & 4.523e-02          & 1.862              & 1.721              & 1.701   \\ 
		$1/64$           & $1/16$  & 9.633e-03          & 1.219e-02          & 1.321e-02          & 1.810              & 1.852              & 1.776   \\ 
		$1/128$          & $1/32$  & 2.899e-03          & 3.743e-03          & 3.981e-03          & 1.733              & 1.703              & 1.730    \\
		\hline\hline
	\end{tabular}%
	\caption{Absolute errors and order of convergence at $t=1$ for Example \ref{ex:2}.}	
	\label{table:ex2} 
\end{table}

\begin{example}\label{ex:2}
	As a second example we consider the time fractional advection-dispersion equation
	\begin{linenomath*}
		\begin{equation*}
		u^{(\alpha)}_t(x,y,t) + \frac{1}{1+x}u_y(x,y,t) +  \frac{1}{1+y}u_y(x,y,t) = \varepsilon\Delta u(x,y,t)  + f(x,y,t), \qquad \varepsilon>0,
		\end{equation*}
	\end{linenomath*}
	on $\Omega = (0,1)\times (0,1)$ for $0\leq t\leq 1$,
	with initial condition
	\begin{linenomath*}
		\begin{equation*}
		u(x,y,0) = \sin\pi x \sin\pi y, \qquad  (x,y)\in\Omega,
		\end{equation*}
	\end{linenomath*}
	boundary condition %$u(x,y,t) = 0$ for $(x,y)\in\partial\Omega\times (0,1]$, and
	\begin{linenomath*}
		\begin{equation*}
		u(x,y,t) = 0, \qquad (x,y)\in\partial\Omega\times (0,1],
		\end{equation*}
	\end{linenomath*}
	and source or sink term
	\begin{linenomath*}
		\begin{equation*}
		f(x,y,t) =  \left( \dfrac{2t^{2-\alpha}}{\Gamma(3-\alpha)} + 2\varepsilon\pi^2(t^2+1) \right) \sin\pi x\sin\pi y 
		+ \pi(t^2+1)\left( \frac{\cos\pi x\sin\pi y}{x+1} + \frac{\cos\pi \sin\pi y}{y+1} \right).
		\end{equation*}
	\end{linenomath*}
	The exact solution is
	\begin{linenomath*}
		\begin{equation*}
		u(x,y,t) = (t^2+1)\sin\pi x \sin\pi y.
		\end{equation*}
	\end{linenomath*}
	% which may be verified by direct substitution in the fractional differential equation.
	
	This is a test for the behavior of the method in the presence of very small diffusion coefficients, an almost 
	degenerate parabolic equation. Numerical results are provided in Table \ref{table:ex2}. As before, 
	three different fractional derivative orders are taken into account and there are results for three different 
	diffusion coefficients: $\varepsilon=10^{-1}$, $\varepsilon=10^{-3}$ and $\varepsilon=10^{-5}$.
	
\end{example}

%\begin{figure}[t]
%		\includegraphics[width=\textwidth]{images/example2a}
%		\caption{$\alpha = 0.3$}  
%		\label{f1} 
%\end{figure}     
%	\begin{figure}[t]
%		\includegraphics[width=\textwidth]{images/example2b}
%		\caption{$\alpha = 0.5$}  
%		\label{f2}      
%	\end{figure}
%	\begin{figure}[t]
%		\includegraphics[width=\textwidth]{images/example2c}
%		\caption{$\alpha = 0.7$}    
%		\label{f3}    
%	\end{figure}
%	\begin{figure}[t]
%		\includegraphics[width=\textwidth]{images/example2d}
%		\caption{$\alpha = 0.9$}     
%		\label{f4}   
%	\end{figure}
%%	\caption{Numerical solutions of Example \ref{ex:3} for $t=1$}\label{fig:ex3}
%%\end{figure}

\begin{figure}[t]
	\centering
	\begin{subfigure}[b]{0.45\textwidth}
		\includegraphics[width=\textwidth]{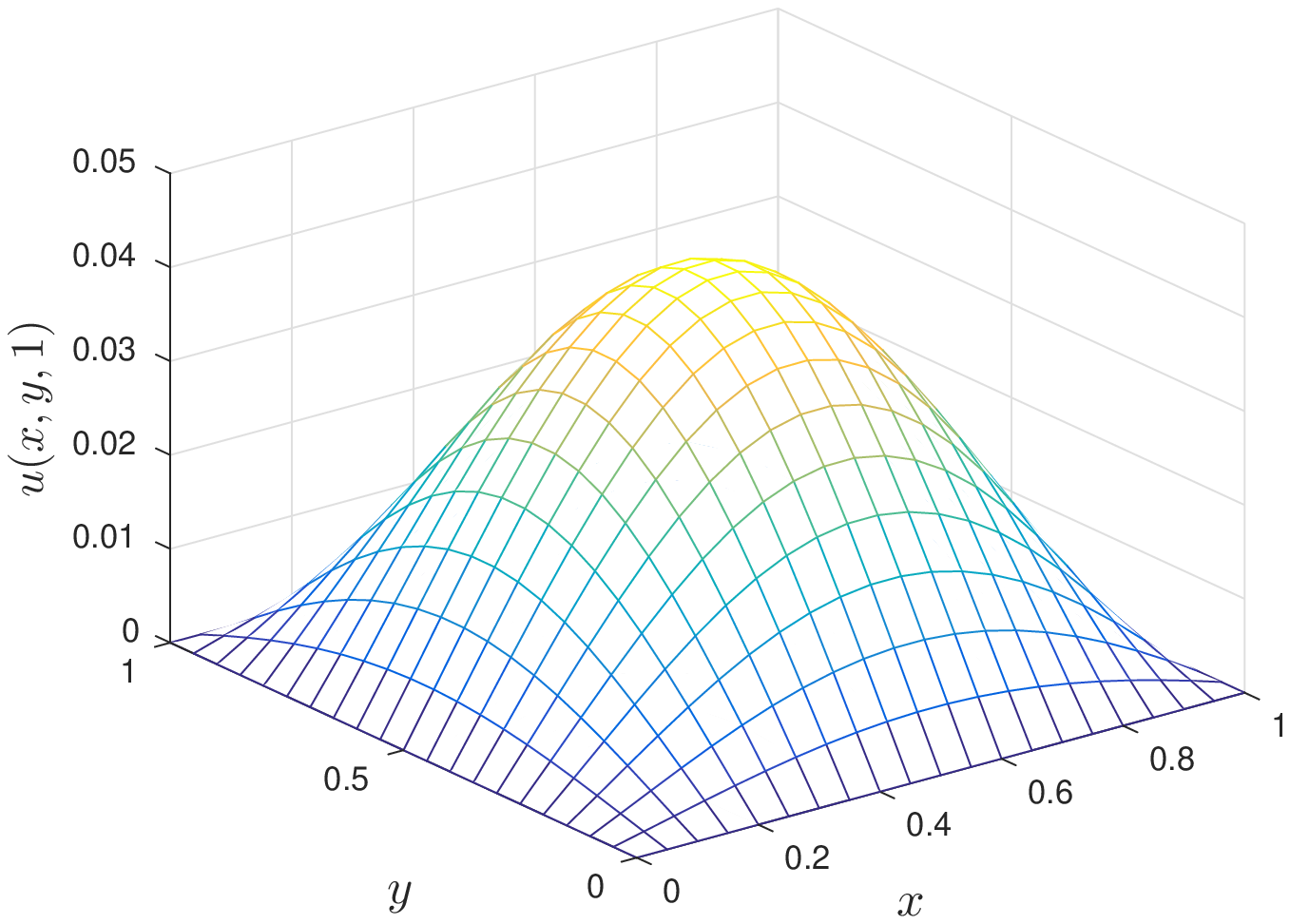}
		\caption{$\alpha = 0.3$}        
	\end{subfigure}
	\hfill
	\begin{subfigure}[b]{0.45\textwidth}
		\includegraphics[width=\textwidth]{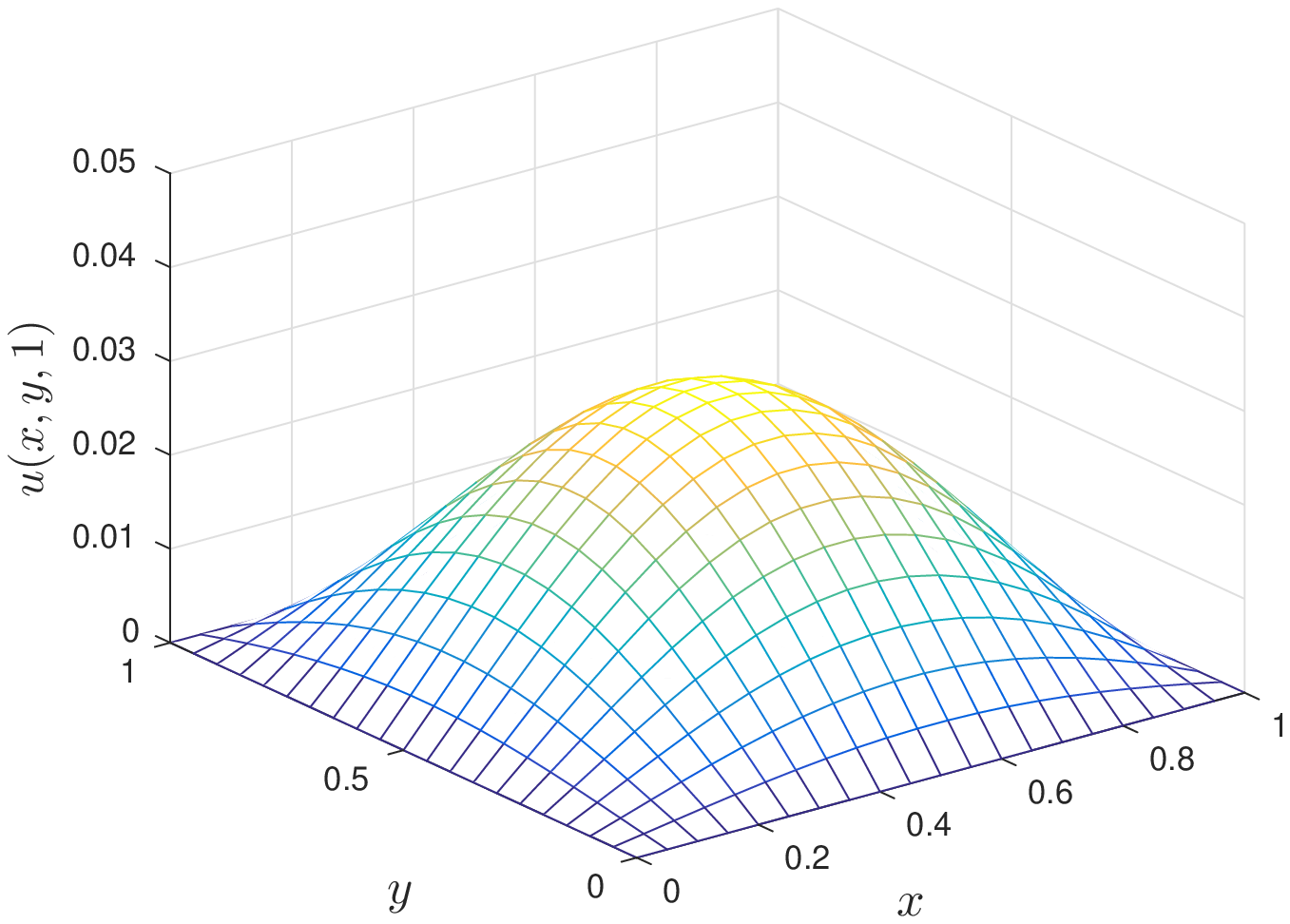}
		\caption{$\alpha = 0.5$}        
	\end{subfigure}
	\\[4mm]
	\begin{subfigure}[b]{0.45\textwidth}
		\includegraphics[width=\textwidth]{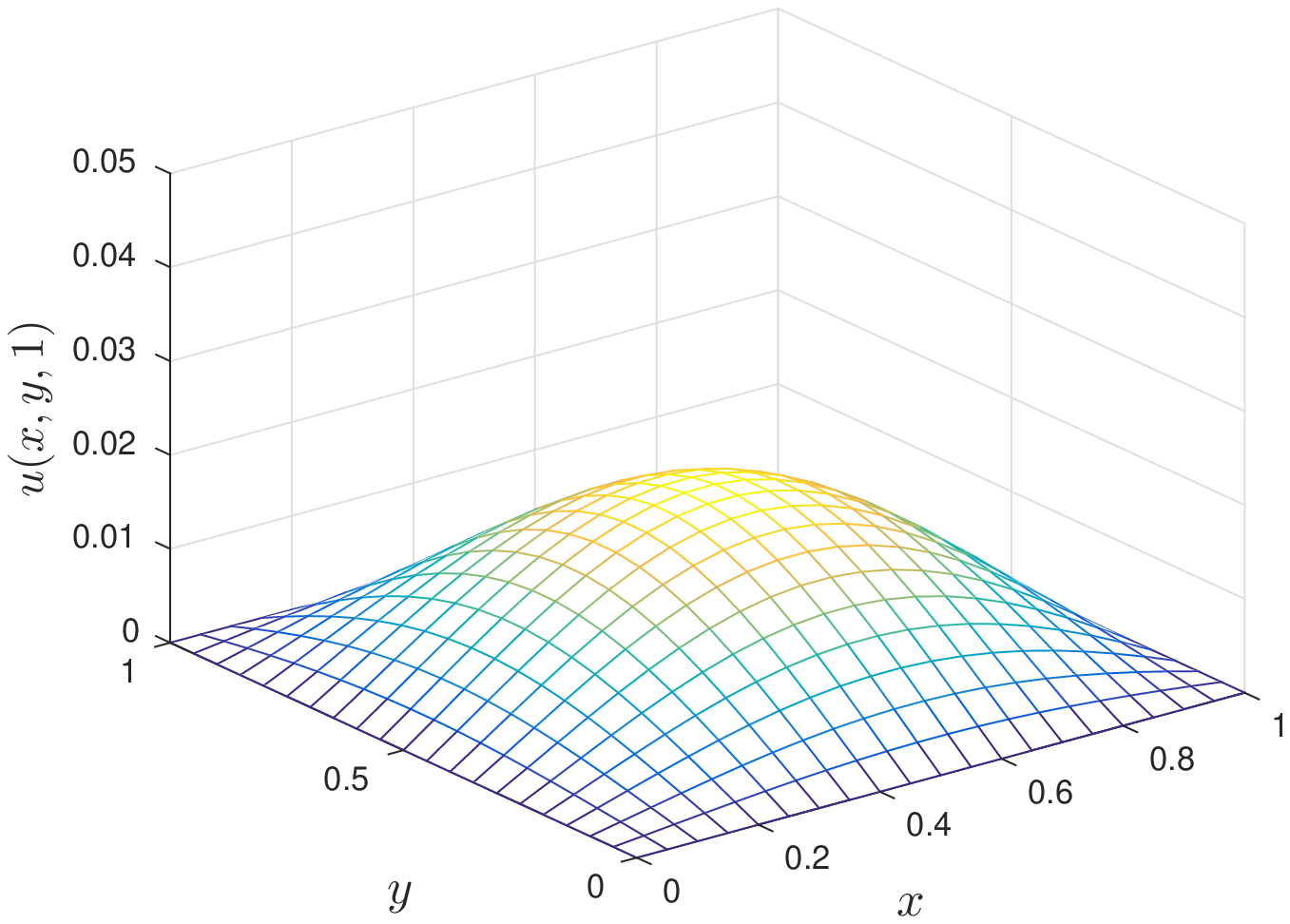}
		\caption{$\alpha = 0.7$}        
	\end{subfigure}
	\hfill
	\begin{subfigure}[b]{0.45\textwidth}
		\includegraphics[width=\textwidth]{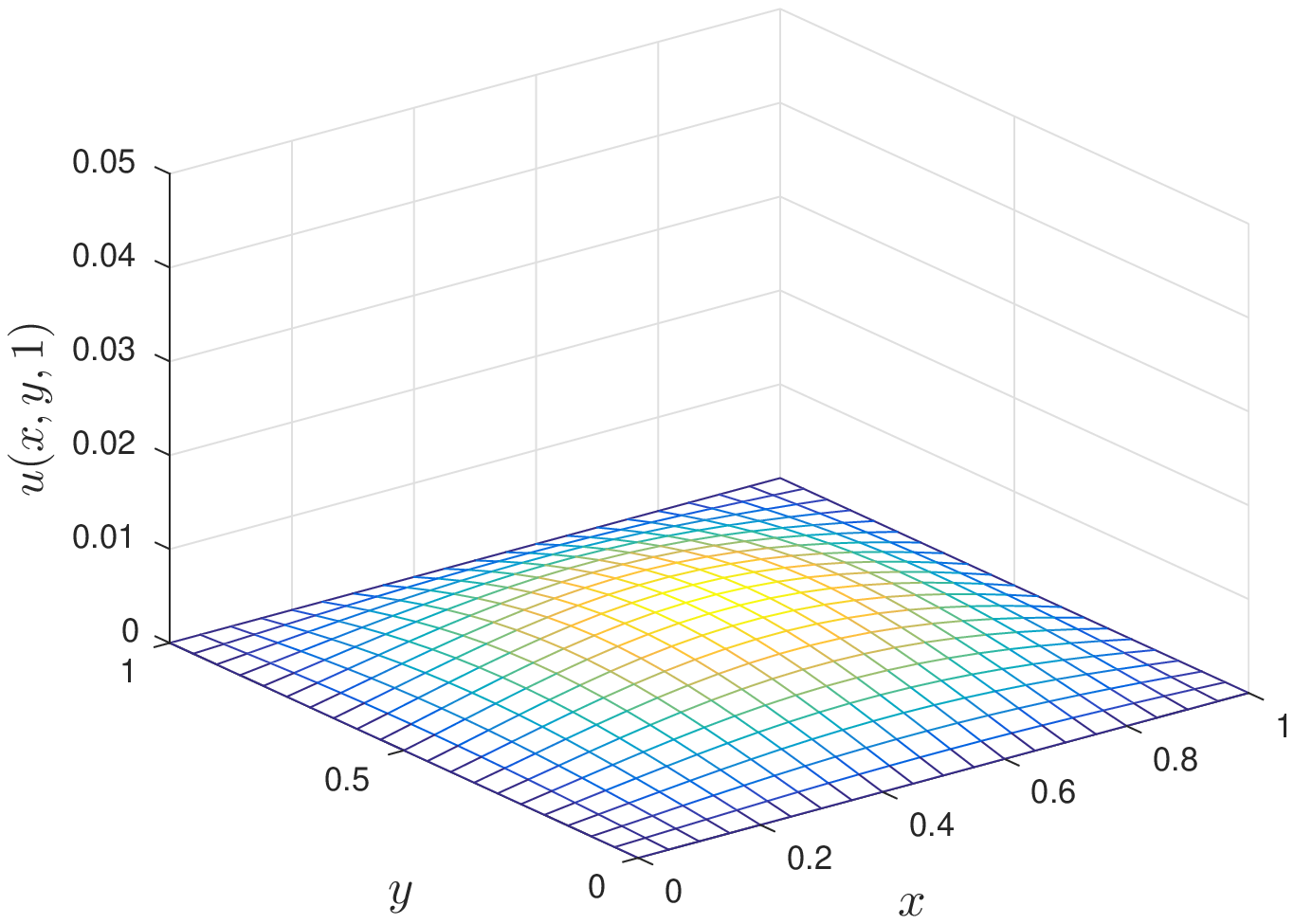}
		\caption{$\alpha = 0.9$}        
	\end{subfigure}
	\caption{Numerical solutions of Example \ref{ex:3} for $t=1$}\label{fig:ex3}
\end{figure}

\begin{example}\label{ex:3}
	Finally we solve the time fractional diffusion equation
	\begin{linenomath*}
		\begin{equation*}
		u_t^{(\alpha)} (x,y,t) + u_x(x,y,t) + u_y(x,y,t) = \Delta u(x,y,t), \\%[2mm] 
		\end{equation*}
	\end{linenomath*}
	on the finite square domain $\Omega = (0,1)\times (0,1)$ for $0\leq t\leq 1$,
	with the initial condition
	\begin{linenomath*}
		\begin{equation*}
		u(x,y,0) = \sin\pi x \sin\pi y, \qquad \mbox{for } (x,y)\in\Omega,
		\end{equation*}
	\end{linenomath*}
	% and the zero Dirichlet boundary condition $u(x,y,t) = 0$ for $(x,y)\in\partial\Omega\times (0,1]$.
	and the boundary condition
	\begin{linenomath*}
		\begin{equation*}
		u(x,y,t) = 0, \qquad  (x,y)\in\partial\Omega\times (0,1].
		\end{equation*}
	\end{linenomath*}
	
	Figure \ref{fig:ex3} illustrates the computed solutions for $t=1$ for several values of $\alpha$. No exact 
	solutions are known for this problem but the pictures illustrate the continuous dependence 
	of the solutions on the fractional differentiation order.
	
\end{example}

In summary, this paper introduces an implicit finite difference approximation for the solution of an initial 
boundary value problem for a two dimensional time fractional advection-dispersion equation 
with variable coefficients in which the fractional derivative is given in the sense of Caputo and the dispersion terms are in nondivergence form. Proofs of consistency, stability and 
convergence are included and so are illustrative numerical experiments. A useful feature of the paper is the  computational framework based on matrices. Our scheme was successfully implemented for the solution of an inverse source problem in \cite{EcheverryMejia2018}. We certainly expect to develop other applications of this scheme in the near future.

% % \clearpage

%
%\section{Final remarks}\label{5}
%
%In this paper we introduce a new implicit finite difference approximation for the solution of an initial 
%boundary value problem for a two dimensional time fractional advection-dispersion equation 
%with variable coefficients. Proofs of consistency, stability and 
%convergence are included and so are illustrative numerical experiments. Extensions of this method to 
%different equations or to support the solution of inverse problems are the subject of current research. 
%
\section*{Acknowledgments}
The authors would like to acknowledge financial support by Universidad Nacional de Colombia through the 
research project with Hermes code 33154.

%\section{References}

\bibliographystyle{plain}
\bibliography{biblio-nov2018}

\begin{thebibliography}{10}

\bibitem{Balasim-Ali2017}
A.T. Balasim and N.H.M. Ali.
\newblock New group iterative schemes in the numerical solution of the
  two-dimensional time fractional advection-diffusion equation.
\newblock {\em Cogent Mathematics}, 4:1412241, 2017.

\bibitem{book:diethelm}
K.~Diethelm.
\newblock {\em The analysis of fractional differential equations}.
\newblock Springer, 2010.

\bibitem{EcheverryMejia2018}
M.~D. Echeverry and C.~E. Mej\'{\i}a.
\newblock A two dimensional discrete mollification operator and the numerical
  solution of an inverse source problem.
\newblock {\em AXIOMS}, 7(4):89, 2018.

\bibitem{Fomin2010}
S.~Fomin, V.~Chugunov, and T.~Hashida.
\newblock Application of fractional differential equations for modeling the
  anomalous diffusion of contaminant from fracture into porous rock matrix with
  bordering alteration zone.
\newblock {\em Transp Porous Med}, 81:187--205, 2010.

\bibitem{book:kilbasST2006}
A~A Kilbas, H~M Srivastava, and J~J Trujillo.
\newblock {\em Theory and applications of fractional differential equations}.
\newblock Elsevier North-Holland, 2006.

\bibitem{LinXu2007}
Y.~Lin and C.~Xu.
\newblock Finite difference/spectral approximation for the time-fractional
  diffusion equation.
\newblock {\em Journal of Computational Physics}, 225:1533--1552, 2007.

\bibitem{MP2017}
C.~E. Mej\'{\i}a and A.~Piedrahita.
\newblock Solution of a time fractional inverse advection-dispersion problem by
  discrete mollification.
\newblock {\em Revista Colombiana de Matem\'aticas}, 51(1):83--102, 2017.

\bibitem{OldhamS2006}
K.B. Oldham and J.~Spanier.
\newblock {\em The fractional calculus: theory and applications of
  differentiation and integration to arbitrary order}.
\newblock Dover, Mineola, 2006.

\bibitem{book:podlubny}
I.~Podlubny.
\newblock {\em Fractional differential equations}.
\newblock Academic Press, 1999.

\bibitem{qiao2017rbf}
Yuanyang Qiao, Shuying Zhai, and Xinlong Feng.
\newblock {RBF-FD} method for the high dimensional time fractional
  convection-diffusion equation.
\newblock {\em International Communications in Heat and Mass Transfer},
  89:230--240, 2017.

\bibitem{WangRen2018}
Y.-M. Wang and L.~Ren.
\newblock Efficient compact finite difference methods for a class of
  time-fractional convection-reaction-diffusion equations with variable
  coefficients.
\newblock {\em International Journal of Computer Mathematics},
  DOI:10.1080/00207160.2018.1437262, 2018.

\bibitem{zhao2017numerical}
Linlin Zhao, Fawang Liu, and Vo~V Anh.
\newblock Numerical methods for the two-dimensional multi-term time-fractional
  diffusion equations.
\newblock {\em Computers \& Mathematics with Applications}, 74(10):2253--2268,
  2017.

\bibitem{ZhuangLiu2007}
Pinghui Zhuang and Fawang Liu.
\newblock Finite difference approximation for two-dimensional time fractional
  diffusion equation.
\newblock {\em Journal of Algorithms \& Computational Technology}, 1(1):1--16,
  2007.

\end{thebibliography}

\end{document}